\documentclass[alphabetic]{elsarticle}

\usepackage{enumerate, longtable}
\usepackage{amsmath, amscd, amsfonts, amsthm, amssymb, latexsym, comment, stmaryrd, graphicx, dsfont, mathtools}
\usepackage{xcolor}
\usepackage[all]{xy}
\usepackage{fullpage}
\usepackage{tikz}
\usepackage[
        colorlinks,
        linkcolor=red,  citecolor=blue,
        backref
]{hyperref}

\usepackage{amstext} 
\usepackage{array}
\usepackage{cleveref}

\newtheorem{claim}{Claim}[section]

\newtheorem{conj}[claim]{Conjecture} 
\newtheorem{thm}[claim]{Theorem}
\newtheorem{prop}[claim]{Proposition}
\newtheorem{cor}[claim]{Corollary}
\newtheorem{lem}[claim]{Lemma}

\theoremstyle{definition}

\newtheorem{rem}[claim]{Remark}

\newtheorem{notation}[claim]{Notation}

\newcommand{\case}[1]{\vspace{0.3cm}\noindent\framebox{#1.}}

\newcommand{\F}{{\mathbb F}}
\newcommand{\Z}{{\mathbb Z}}
\newcommand{\Q}{{\mathbb Q}}

\renewcommand{\P}{{\mathbb P}}

\newcommand{\con}{{\mathrm{con}}}
\newcommand{\mon}{{\mathrm{mon}}}
\newcommand{\Gal}{{\mathrm{Gal}}}
\newcommand{\disc}{{\mathrm{disc}}}

\DeclarePairedDelimiter{\floor}{\lfloor}{\rfloor}

\numberwithin{equation}{section}
\begin{document}

\title{Probabilistic Galois Theory in Function Fields}



\author[inst1]{Alexei Entin}
\ead{aentin@tauex.tau.ac.il}

\affiliation[inst1]{organization={School of Mathematical Sciences, Tel Aviv University},
            addressline={Tel Aviv University}, 
            city={Tel Aviv},
            postcode={69978}, 
            country={Israel}}

\author[inst1]{Alexander Popov}
\ead{alexanderp@mail.tau.ac.il}
\journal{arXiv}
\date{}

\begin{abstract}

We study the irreducibility and Galois group of random polynomials over function fields. We prove that a random polynomial $f=y^n+\sum_{i=0}^{n-1}a_i(x)y^i\in\F_q[x][y]$ with i.i.d coefficients $a_i$ taking values in the set $\{a(x)\in\mathbb{F}_q[x]: \deg a\leq d\}$ with uniform probability, is irreducible with probability tending to $1-\frac{1}{q^d}$ as $n\to\infty$, where $d$ and $q$ are fixed. We also prove that with the same probability, the Galois group of this random polynomial contains the alternating group $A_n$. Moreover, we prove that if we assume a version of the polynomial Chowla conjecture over $\mathbb{F}_q[x]$, then the Galois group of this polynomial is actually equal to the symmetric group $S_n$ with probability tending to $1-\frac{1}{q^d}$. We also study the other possible Galois groups occurring with positive limit probability. Finally, we study the same problems with $n$ fixed and $d\to\infty$. 

\end{abstract}

\maketitle

\section{Introduction and main results}

Probabilistic Galois theory studies statistical properties of the Galois group of a random polynomial drawn from some natural ensemble of polynomials over a global field. The classical case is the field of rational numbers $\Q$, where the random polynomials are taken with integer coefficients. Intuitively, a random polynomial of degree $n$ is almost surely irreducible over $\mathbb{Q}$, and moreover, its Galois group over $\mathbb{Q}$ is almost surely isomorphic to $S_n$. Of course, to state this formally we first need to define what we mean by a random polynomial. There are several models studied in the literature.

The simplest model is called the large box model. We fix some $n\geq 1$, 
and for a given $H\in\mathbb{N}$ we choose a random monic polynomial 
$f\in\mathbb{Z}[x]$ of degree $n$ with independent coefficients drawn 
uniformly from $\{-H,\ldots,H\}$. We keep $n$ fixed, and take $H\to\infty$. It 
was proved a century ago by Hilbert (as a special case of his famous 
irreducibility theorem) that $$\lim\limits_{H\to\infty}\mathbb{P}(\text{$f$ 
is irreducible over $\mathbb{Q}$})=1.$$ Moreover, if $f$ is separable (this 
is automatic if it is irreducible) then we may consider its Galois group 
(over $\mathbb{Q}$), which we denote by $G_f$, equipped with an action on 
the $n=\deg f$ roots of $f$. It was proved by van der Waerden 
\cite{VDW36} that \begin{equation}\label{eq: van der 
waerden}\lim\limits_{H\to\infty}\mathbb{P}(G_f=S_n)=1.\end{equation} 
Moreover, van der Waerden conjectured that in fact $$\mathbb{P}
(G_f=S_n)=O(H^{-1}),$$ where $n$ is fixed and $H\to\infty$. Recently this conjecture was proved by Bhargava \cite{Bha21_}. 

Another natural model is the small box model (also called 
the restricted coefficients model), in which we fix the 
set of possible coefficients and take the degree to 
infinity. In \cite{OdPo93} Odlyzko and Poonen 
conjectured that a random polynomial $P(x)=x^n+a_{n-
1}x^{n-1}+\ldots+a_1x+1\in\mathbb{Z}[x]$ with independent 
coefficients $a_1,\ldots,a_{n-1}$ taking values in $\{0,1\}$ 
uniformly is irreducible converges to $1$ as 
$n\to\infty$. Konyagin \cite{Kon99}
proved that $$\mathbb{P}(\text{$P(x)$ is 
irreducible})>\frac{c}{\log n}$$ for some constant $c>0$. 
This was later improved to $$\mathbb{P}(\text{$P(x)$ is 
irreducible})>c$$ by Bary-Soroker, Koukoulopoulos and Kozma \cite{BKK23}. 

In \cite{BaKo20} Bary-Soroker and Kozma proved that if we pick random monic polynomial $f\in\mathbb{Z}[x]$ of degree $n$ with coefficients in the set $\{1,2,\ldots,210\}$ then:
\begin{equation}\label{eqn:(1,1)}\lim\limits_{n\to\infty}\mathbb{P}(\text{$f$ is irreducible over $\mathbb{Q}$})=1.\end{equation}
Moreover:
\begin{equation}\label{eqn:(1,2)}\lim\limits_{n\to\infty}\mathbb{P}(\text{$G_f=S_n$ or $A_n$})=1.\end{equation}
One can replace $210$ by any integer divisible by at least $4$ distinct prime numbers. In fact, this is used only in order to prove the statement about the irreducibility of $f$. Given that $f$ is irreducible we almost surely have $G_f=A_n$ or $S_n$, this part can be proved if we replace $210$ by any integer $L\geq 2$. The proof of the latter also appears in \cite{BaKo20}. 

In \cite{BrVa19} Breuillard and Varj{\'{u}} proved that assuming the Extended Riemann Hypothesis, \eqref{eqn:(1,1)} and \eqref{eqn:(1,2)} hold if we replace $210$ by any integer $L\geq 2$. Their methods involve studying the roots of random polynomials over a finite field via random walks, and then using this information to study irreducibility of random polynomials over $\mathbb{Z}$ and their Galois group. In \cite{BKK23} Bary-Soroker, Kozma and Koukoulopoulos proved that the same result holds for any $L\geq 35$ unconditionally. 

In the present work we study the analogous problems over the rational function field $\mathbb{F}_q(x)$, where $\mathbb{F}_q$ is a finite field with $q$ elements. We note that the Galois theory of random \emph{additive} polynomials over $\F_q(x)$ was studied by the first author, Bary-Soroker and McKemmie \cite{BEM24}, but here we consider the more obvious model of a random polynomial. From now on suppose that $q$ is a prime power. We consider both the large box model and the small box model. First we state our main result for the large box model, an analog of (\ref{eq: van der waerden}) which will be easily derived from the function field version of Hilbert's Irreducibility Theorem.

\begin{thm}\label{1.1}\label{thm: large box}Let $n$ be a fixed natural number and $q$ a fixed prime power. We choose a random monic polynomial of the form:
$$f(x,y)=y^n+a_{n-1}(x)y^{n-1}+\ldots+a_1(x)y+a_0(x)\in\mathbb{F}_q[x][y],$$
where $a_0, a_1,\ldots,a_{n-1}\in\mathbb{F}_q[x]$ are polynomials of degree at most $d$, chosen uniformly and independently at random. When $f$ is separable, denote by $G_f$ its Galois group over $\mathbb{F}_q(x)$. Then
$$\lim\limits_{d\to\infty}\mathbb{P}(\text{$f$ is irreducible and separable over $\mathbb{F}_q(x)$, $G_f=S_n$})=1.$$
\end{thm}

In the small box model we fix $d$ and take $n\to\infty$. This model is more complicated. Unlike over $\mathbb{Z}$, here the polynomial is not almost surely irreducible in the small box model, even if we condition on $a_0\neq 0$. See \Cref{rem: content} below for more details on this interesting phenomenon. First we consider the question of irreducibility.

\begin{thm}\label{1.2}\label{thm: irreducibility} Let $d$ be a fixed natural number and $q$ a fixed prime power. We choose a random monic polynomial of the form
\begin{equation*}
f(x,y)=y^n+a_{n-1}(x)y^{n-1}+\ldots+a_1(x)y+a_0(x)\in\mathbb{F}_q[x][y],
\end{equation*}

\noindent where $a_0, a_1,\ldots,a_{n-1}$ are polynomials of degree at most $d$ drawn uniformly and independently at random. Then
$$\lim\limits_{n\to\infty}\mathbb{P}(\text{$f$ is irreducible})=1-\frac{1}{q^d}.$$
\end{thm}

\begin{rem} A variant of this result with leading term $a_n(x)y^n,\,\deg a_n\le d$ in place of $y^n$ was proved by Carlitz \cite{Car65}. The same limit probability $1-\frac 1{q^d}$ is obtained in that model. Theorem \ref{thm: irreducibility} was proved independently of Carlitz's work (of which we learnt only after completing the present work) and by a rather different method. Both results can be derived by either method, though our method is more robust to further variations.\end{rem} 

Next we consider the distribution of the Galois group in the small box model, conditional on $f$ being irreducible. This part is based on the methods of \cite{BaKo20}.

\begin{thm}\label{1.3}\label{thm: small box main} Let $d$ be a fixed natural number and $q$ a fixed prime power. For $n\geq 1$, consider a random monic polynomial of the form
\begin{equation*}
f(x,y)=y^n+a_{n-1}(x)y^{n-1}+\ldots+a_1(x)y+a_0(x)\in\mathbb{F}_q[x][y],
\end{equation*}
where $a_0, a_1,\ldots,a_{n-1}$ are polynomials of degree at most $d$, chosen uniformly and independently at random. If $f$ is separable denote by $G_f$ its Galois group over $\mathbb{F}_q(x)$. Then
$$\lim\limits_{n\to\infty}\mathbb{P}(\text{$G_f=S_n$ or $A_n$}\ |\ \text{$f$ is irreducible})=1.$$
\end{thm}

Theorem \ref{1.3} only tells us that the Galois group of $f$ is almost surely $S_n$ or $A_n$, given that $f$ is irreducible over $\mathbb{F}_q(x)$. Distinguishing between $S_n$ and $A_n$ has proven to be a formidable challenge. Currently it is not known how to do this in the small box model over $\mathbb{Z}$, and neither can we do it over $\mathbb{F}_q[x]$. However, we can do this conditionally on a uniform version of a standard open conjecture, namely the polynomial Chowla conjecture (we note that a similar conditional result is not known over $\Z$ and it is not clear how to obtain one). Before we state this variant, recall the definition of the Liouville $\lambda$-function: $\lambda(g)=(-1)^k$, where $g=P_1\cdots P_k$ is the prime factorization of $g\in\F_q[x]$. A polynomial defined over a unique factorization domain is called \emph{content-free} if the GCD of its coefficients is 1.

\begin{conj}\label{1.4} \label{conj: chowla}[Uniform version of the polynomial Chowla Conjecture] Let $q$ be a fixed power of an odd prime, $c>0$ a fixed constant. Then for any separable (in the variable $T$) polynomial $F\in\F_q[x][T]$ we have
$$\sum\limits_{h\in\F_q[y]\,\mathrm{monic}\atop{\deg h=n}} \lambda(F(y,h(y)))=o(q^n),\quad n\to\infty$$
uniformly in $F$ with $\deg_y(F)\le c n,\,\deg_T(F)\le c $. 
\end{conj}

Sawin and Shusterman \cite{SaSh22} proved a weaker version of the polynomial Chowla Conjecture over $\mathbb{F}_q[x]$ 
\footnote{Sawin and Shusterman state the conjecture and their results with the M\"obius function $\mu$ in place of the Liouville function $\lambda$. In the function field setting this can be shown to be equivalent using the results of \cite{Car21}. Chowla's original conjecture is stated in terms of $\lambda$ \cite[p. 96]{Cho65}}.
Unfortunately it is insufficient for our next result, which requires the full Conjecture \ref{1.4} (in their result $q$ has to grow with $c$). 
Assuming Conjecture \ref{1.4} we can improve the result of Theorem \ref{1.3} and show that the Galois group of $f$ is almost surely isomorphic to $S_n$. 
\begin{thm}\label{1.5}\label{thm: chowla} Let $d$ be a fixed natural number and $q$ a fixed \emph{odd} prime power. Assume Conjecture \ref{1.4} holds. Let
\begin{equation*}
f(x,y)=y^n+a_{n-1}(x)y^{n-1}+\ldots+a_1(x)y+a_0(x)\in\mathbb{F}_q[x][y],
\end{equation*}
be a random polynomial, where $a_0, a_1,\ldots,a_{n-1}\in\F_q[x]$ are polynomials of degree at most $d$, chosen uniformly and independently at random. If $f$ is separable denote by $G_f$ its Galois group. Then
$$\lim\limits_{n\to\infty}\mathbb{P}(\text{$G_f=S_n$}\ |\ \text{$f$ is irreducible over $\mathbb{F}_q(x)$})=1.$$
\end{thm}

\begin{rem} The reason we need to assume that $q$ is odd in the proof of Theorem \ref{thm: chowla} is that we make use of the discriminant criterion (for the Galois group being contained in $A_n$) and Pellet's formula (Lemma \ref{lem: pellet} below), both of which need to be modified in characteristic 2. We expect the assertion to be valid without this assumption and our general proof strategy to apply, but the details of the argument would differ significantly. The same applies to Theorem \ref{thm: reducible chowla} below.\end{rem}

We also study the case when $f$ is not irreducible. In this case there is a positive probability for $f$ not being separable. However, if the irreducible factors of $f$ are separable then the splitting field of $f$ is usually still a Galois extension of $\mathbb{F}_q(x)$, and we may consider the Galois group of this extension. Recall that the content of a polynomial with coefficients in a unique factorization domain is the GCD of its coefficients, which we always take to be monic. A polynomial $f\in\F_q[x][y]$ can be viewed also as a polynomial in $\F_q[y][x]$ and we denote by $\con_x(f)\in\F_q[y]$ its content as a polynomial over $\F_q[y]$. As we will see below, the possibility of $\con_x(f)\neq 1$ is the main source of reducible values of $f$ in the small box model (this function field phenomenon has no analog over the integers). The following result extends \Cref{1.3} to the reducible case.

\begin{thm}\label{1.6}\label{thm: reducible} Let $d$ be a fixed natural number and $q$ a fixed prime power. Let
$$f(x,y)=y^n+a_{n-1}(x)y^{n-1}+\ldots+a_1(x)y+a_0(x)\in\mathbb{F}_q[x][y]$$ be a random polynomial,
where $a_0, a_1,\ldots,a_{n-1}\in\mathbb{F}_q[x]$ are polynomials of degree $\le d$ chosen uniformly and independently at random. We can uniquely write $f(x,y)=c(y)g(x,y)$ where $c=\con_x(f)\in\mathbb{F}_q[y]$ and $g(x,y)\in\mathbb{F}_q[x][y]$ is monic in $y$ with $\con_x(g)=1$. Then with probability tending to $1$ as $n\to\infty$, the splitting field of $f$ over $\mathbb{F}_q(x)$ is a Galois extension of $\mathbb{F}_q(x)$. Further, if we denote the Galois group of the splitting field of $c(y)$ over $\mathbb{F}_q$ by $C$ (note that it is cyclic since $\F_q$ is a finite field), then with probability tending to $1$ as $n\to\infty$ the Galois group of the splitting field of $f$ over $\mathbb{F}_q(x)$ is isomorphic to one of the following groups:
\begin{enumerate}
\item[(a)] 
$C\times S_{\deg g}$.
\item[(b)]  
$C\times A_{\deg g}$.
\item[(c)] $C\times_{C_2}S_{\deg g}:=\{(\sigma,\tau)\in C\times S_{\deg g}:\phi(\sigma)=\psi(\tau)\}$, where $C$ is of even order, $C_2$ is the cyclic group of order 2 and $\phi:C\to C_2,\,\psi:S_{\deg g}\to C_2$ are the unique epimorphisms from $C,S_{\deg g}$ to $C_2$.
\end{enumerate}
\end{thm}

As in \Cref{1.3}, we are unable to distinguish between the three cases (a),(b),(c) unconditionally. However, assuming Conjecture \ref{1.4} there is just one possibility for the Galois group occuring with positive limit probability. 

\begin{thm}\label{1.7}\label{thm: reducible chowla} Assume Conjecture \ref{1.4}. In the setup of Theorem \ref{1.6} assume additionally that $q$ is odd. Then with probability tending to 1 as $n\to\infty$ the irreducible factors of $f$ are separable and the Galois group of the splitting field of $f$ over $\F_q(x)$ is isomorphic to $C\times S_{\deg g}$.
\end{thm}

To describe the distribution of the Galois groups more precisely, we supplement Theorems \ref{1.6} and \ref{1.7} with the following proposition (proved at the end of Section \ref{sec:prelim}) regarding the distribution of the number $\deg g$ and the order of the group $C$.

\begin{prop}\label{prop: content} In the setup of Theorem \ref{thm: reducible}, for any $0\le k\le n-1$ and monic $c\in\F_q[y],\deg c=k$ we have
$$\P(\con_x(f)=c,\,\deg g=n-k)=\left(1-\frac 1{q^d}\right)\frac 1{q^{k(d+1)}}$$ (note that we have an exact equality and not a limit).
It follows from the elementary theory of finite fields that the order of $C=\Gal(c(y)/\F_q[x])=\Gal(c(y)/\F_q)$ is $|C|=\mathrm{lcm}(\deg P_1,\ldots,\deg P_m)$, where $c=\prod_{i=1}^mP_i^{e_i}$ is the prime factorization of $c$.
\end{prop}

The paper is organized as follows: in Section \ref{sec:prelim} we introduce some notation and prove some results on the probability of coprimality in $\F_q[x]$. In Section \ref{sec: large box} we discuss the large box model and prove Theorem \ref{thm: large box}. In section \ref{sec: small box irred} we begin investigating the small box model and deduce Theorem \ref{thm: irreducibility} from the large box model using an exchange of variables. In section \ref{sec: small box galois} we study the Galois group in the small box model and prove Theorem \ref{thm: small box main}. In Section \ref{sec: chowla} we derive Theorem \ref{thm: chowla}. Finally, in Section \ref{sec: reducible} we extend our results to the reducible case, proving Theorems \ref{thm: reducible} and \ref{thm: reducible chowla}. 
\\ \\
\noindent {\bf Acknowledgements.} The authors would like to thank Ofir Gorodetsky for bringing to our attention some related work and Evan O'Dorney for pointing out several typos and some suggestions on the exposition. The authors would also like to thank the anonymous referee of a previous version of this paper for spotting several minor errors and for suggestions for improving the exposition. Both authors were partially supported by a grant of the Israel Science Foundation no. $2507/19$. 

\section{Preliminaries}\label{sec:prelim}
\subsection{Notaion and conventions}
The letter $q=p$ will always denote a power of a prime $p$, which will always be assumed \emph{fixed}. In particular all asymptotic notation has implicit constants and rate of convergence which may depend on $q$ as well as other parameters explicitly specified as fixed. This is just for convenience, in principle all of our results can be made uniform in $q$. 

The cardinality of a set $A$ will be denoted by both $|A|$ and $\#A$. Non-strict inclusion is denoted by $\subset$. If $A$ is an event in some given sample space we denote by $A^c$ its complement. For a polynomial ring $R[X]$ we denote
$$R[X]_{\le d}=\{h\in R[X]:\,\deg h\le d\},$$
$$R[X]_{\le d}^n=\{(h_1,\ldots h_n)\in R[X]^n:\,\deg h_i\le d\},$$
$$R[X]^{\mon}=\{h\in R[X]:\,h\mbox{ monic}\},$$
$$R[X]^{\mon}_{\le d}=\{h\in R[X]:\,h\mbox{ monic, }\deg h\le d\},$$
$$R[X]^{\mon}_{d}=\{h\in R[X]:\,h\mbox{ monic, }\deg h=d\}.$$

\subsection{Coprimality of polynomials over a finite field}

In this section we will prove some auxiliary results about the probability of polynomials in $\F_q[x]$ being coprime, which will be needed in the sequel. For the analogous results over $\Z$ see \cite{Nym72}. 

\begin{lem}\label{2.3}\label{lem: gcd} Let $n_1\geq n_2\geq\ldots\geq n_m$ be natural numbers. Let $A=A^{(1)}\times A^{(2)}\times\cdots\times A^{(m)}$, where $$A^{(i)}\in\{\F_q[x]_{n_i},\,\F_q[x]_{\le n_i-1},\,\F_q[x]_{n_i}^{\mon}\},\quad 1\le i\le m-1,$$
$$A^{(m)}\in\{\F_q[x]_{n_m},\,\F_q[x]_{n_m}^{\mon}\}.$$
Let $(a_1,\ldots,a_m)$ be chosen uniformly at random from $A$. Then 
$$\P(\gcd(a_1,\ldots,a_m)=1)=1-\frac{1}{q^{m-1}}.$$ 
\end{lem}
\begin{proof} 
We will use a standard sieving argument involving the M\"obius function. Recall that the M\"obius function on $\F_q[x]$ is defined by
\begin{equation}\label{eq: mobius}\mu(g)=\left[\begin{array}{ll}(-1)^k,&g=P_1\cdots P_k,\,P_i\in\F_q[x]\mbox{ distinct primes},\,k\ge 0,\\0,&g\mbox{ not squarefree}.\end{array}\right.\end{equation} It is well-known that
\begin{equation}\label{eq: M(n)} M(n):=\sum\limits_{\substack{\deg h =n \\ {\text{$h$ monic}}}}\mu(h)=\left\{\begin{array}{lll}1,&n=0,\\-q,&n=1,\\0,&n>1,\end{array}\right.\end{equation}
see \cite[\S 2]{Ros02} for details.

Consider the set
$$B=\{(a_1,\ldots,a_m)\in A: \gcd(a_1,\ldots,a_m)\ne 1\}.$$
Also, for each polynomial $f\in\mathbb{F}_q[x]$ define 
$$A_f=\{(a_1,\ldots,a_m)\in A: f|\gcd(a_1,\ldots,a_m)\}.$$ If $p_1,\ldots,p_t$ are all the monic irreducible polynomials of degree at most $n_m$ then clearly $B=\bigcup\limits_{i=1}^t A_{p_i}$, $A_{p_{i_1}}\cap A_{p_{i_1}}\cap\ldots\cap A_{p_{i_k}}=A_{p_{i_1}p_{i_2}\cdots p_{i_k}}$ ($i_1,\ldots,i_k$ distinct) and we can compute the cardinality of $B$ by inclusion-exclusion:
$$|B|=\sum\limits_{i} |A_{p_i}|-\sum\limits_{i,j} |A_{p_ip_j}|+\sum\limits_{i,j,k} |A_{p_ip_jp_k}|-\ldots+(-1)^{t+1} |A_{p_1p_2\ldots p_t}|=-\sum\limits_{\substack{\deg h\leq n_m \\ \text{$h$ monic}}}\mu(h)|A_h|$$
(note that $|A_h|=0$ if $\deg h >n_m$). Now, if $\deg h\leq n_m$ 
then $$|A_h|=\frac{|A|}{q^{m\deg h}}.$$ Hence
$$|B|=-\sum\limits_{\substack{\deg h \leq n_m \\ \text{$h$ monic}}}\mu(h)|A_h|=-\sum\limits_{k=1}^{n_m}\sum\limits_{\substack{\deg h =k \\ \text{$h$ monic}}}\mu(h)\frac{|A|}{q^{mk}}=-\sum\limits_{k=1}^{n_m}M(k)\frac{|A|}{q^{mk}}=\frac{|A|}{q^{m-1}},$$
the last equality follows from (\ref{eq: M(n)}). We obtain $\P(B)=\frac{|B|}{|A|}=\frac 1{q^{m-1}}$, which is equivalent to the assertion of the lemma.

\end{proof}

\begin{cor}\label{2.1}\label{cor: gcd} Let $n,d\ge 1$ be integers. Define the following set:
$$A=\left\{(a_0,\ldots,a_d)\in\mathbb{F}_q[x]^{d+1}:\,a_i\in\F_q[x]_{\le n-1},\,a_0\in\F_q[x]_n^{\mon}\right\},$$ 
If we pick (uniformly) a random element $(a_0,\ldots,a_d)\in A$ then $$\mathbb{P}(\gcd(a_0,\ldots,a_d)=1)=1-\frac{1}{q^{d}}.$$

\end{cor}

\begin{proof} 
This is immediate from Lemma \ref{lem: gcd} applied with $m=d+1, n_1= n_2=n_3=\ldots=n_{d+1}=n$.
\end{proof}

\begin{proof}[Proof of Proposition \ref{prop: content}]

In the notation of Proposition \ref{prop: content} the condition $\con_x(f)=c,\deg g=n-k$ is equivalent to $f$ being of the form
$$f=c(y)\cdot\left(b_d(y)x^d+b_{d-1}(y)x^{d-1}+\ldots+b_0(y)\right),$$ $$b_i\in\F_q[y]_{\le n-k-1}\,(1\le i\le d),\,b_0\in\F_q[y]_{n-k}^{\mathrm{mon}},\,\gcd(b_0,\ldots,b_{d})=1.$$
Noting that $$\bigl|\left(\F_q[y]_{\le n-k-1}\right)^d\times\F_q[y]_{n-k}^\mon\bigr|=\frac 1{q^{k(d+1)}}\bigl|\left(\F_q[y]_{\le n-1}\right)^d\times\F_q[y]_{n}^\mon\bigr|$$
The assertion now immediately follows from Corollary \ref{cor: gcd}.
\end{proof}

\section{The large box model}\label{sec: large box}
In the present section we treat the large box model and prove Theorem \ref{1.1}. Our main tool is a quantitative form of Hilbert's Irreducibility Theorem over function fields, which was proved by Bary-Soroker and the first author in \cite{BaEn21}:

\begin{thm}[Quantitative Hilbert Irreducibility Theorem over function fields]\label{3.1}\label{thm: HIT}  Let $x,y,T_1,\ldots,T_n$ be independent variables and $F(T_1,\ldots,T_n,y)\in \F_q(x)[T_1,\ldots,T_n,y]$ an irreducible polynomial with $\deg_yF\ge 1$, separable in $y$. 
Denote $$H_F(d)=\#\{(t_1,\ldots,t_n)\in\left(\F_q[x]_{\le d}\right)^n:\,F(t_1,\ldots,t_n,y)\in K[y]\mbox{ is irreducible over } \F_q(x)\}.$$ Then:
$$\frac{H_F(d)}{q^{(d+1)n}}\to 1 \ \ \ \ \text{as} \ d\to\infty.$$
Moreover, if we denote
\begin{multline*}I_F(d)=\#\left\{(t_1,\ldots,t_n)\in\left(\F_q[x]_{\le d}\right)^n:\,F(t_1,\ldots,t_n,y)\in\F_q(x)[y]\mbox{ separable, }\right.\\ \left.\Gal(F(t_1,\ldots,t_n,y)/\F_q(x))\cong\Gal(F/\F_q(x,T_1,\ldots,T_n))\right\}.\end{multline*}
then
$$\frac{I_F(d)}{q^{(d+1)n}}\to 1 \ \ \ \ \text{as} \ d\to\infty.$$
\end{thm}

\begin{rem} In \cite{BaEn21} the above result is stated with $t_i$ required to be monic. This is equivalent to the statement above by change of variables $s_i= T^{d+1}+t_i$.\end{rem}

\begin{prop}\label{3.2} Let $K$ be a field, $T_1,\ldots,T_n$ independent variables. The polynomial $$F(y)=y^n+T_ny^{n-1}+T_{n-1}y^{n-2}+\ldots+T_2y+T_1\in K(T_1,\ldots,T_n)[y]$$ has Galois group isomorphic to $S_n$ over $K(T_1,\ldots,T_n)$. \end{prop}
\begin{proof} See \cite[Theorem 18.13]{Ste94}. \end{proof}

\begin{proof}[Proof of Theorem \ref{1.1}]  By proposition \ref{3.2} we know that the generic polynomial
$$F(T_1,\ldots,T_n,y)=y^n+T_ny^{n-1}+\ldots+T_2y+T_1\in \F_q(x,T_1,\ldots,T_n)[y]$$
has Galois group $S_n$ over $\F_q(x,T_1,\ldots,T_n)$. For $d\in\mathbb{N}$ let 
$I_F(d)$ be the number of tuples\\ $(a_0(x), a_1(x),\ldots,a_{n-1}
(x))\in\left(\F_q[x]_{\le d}\right)^n$ such that 
\begin{equation}\label{eq: proof large box}F(a_0(x), a_1(x),\ldots,a_{n-1}(x), 
y)=y^n+a_{n-1}(x)y^{n-1}+\ldots+a_1(x)y+a_0(x)\in\F_q(x)[y]\end{equation} has 
Galois group $S_n$ over $\mathbb{F}_q(x)$. Since the Galois group of 
$F(T_1,\ldots,T_n,y)$ over $K(T_1,\ldots,T_n)$ is $S_n$, Theorem \ref{3.1} tells us 
that $\frac{I_F(d)}{q^{(d+1)n}}\to 1$ as $d\to\infty$. Hence with limit 
probability 1 as $d\to\infty$, the polynomial on the RHS of (\ref{eq: proof large 
box}) has Galois group $S_n$ and in particular is irreducible (since $S_n$ is 
transitive). This completes the proof of \Cref{thm: large box}.\end{proof}

\section{The small box model - irreducibility}
\label{sec: small box irred}

For the rest of the paper we focus on the small box model, which is more challenging (as in the classical setting over $\Z$). In the present section we deal with the question of irreducibility and prove \Cref{thm: irreducibility}. It turns out that in our function field setting this question is relatively easy and the main tool will once again be Hilbert's Irreducibility Theorem.

First, let us recall the small box model. We consider a random polynomial
\begin{equation}\label{eq: rand pol 1}f(x,y)=y^n+a_{n-1}(x)y^{n-1}+\ldots+a_1(x)y+a_0(x)\in\mathbb{F}_q[x][y],\end{equation}
where $a_0, a_1,\ldots,a_{n-1}\in\mathbb{F}_q[x]_{\le d}$ are drawn uniformly and independently at random. Throughout this section $q,d\ge 1$ are fixed and $n$ will be taken to infinity.

First of all we sketch the idea of the proof of Theorem \ref{1.2}. For the random polynomial (\ref{eq: rand pol 1}), for each $0\leq i\leq n-1$ write $a_i(x)=\sum\limits_{j=0}^d a_{ij}x^j$. Then
$$f=y^n+\sum\limits_{i=0}^{n-1}\left(\sum\limits_{j=0}^d a_{ij}x^j\right)y^i=\sum\limits_{j=1}^d\left(\sum\limits_{i=0}^{n-1}a_{ij}y^i\right)x^j+\left(y^n+a_{n-1, 0}y^{n-1}+\ldots+a_{10}y+a_{00}\right).$$
So now we regard $f$ as a polynomial in the variable $x$ over $\mathbb{F}_q[y]$. Note that by Gauss's lemma, $f$ is irreducible over $\F_q(x)$ iff it is irreducible as an element of $\F_q[x,y]$, which is equivalent to $\con_x(f)=1$ and $f$ being irreducible over $\F_q(y)$. As a polynomial in $x$, $f$ has degree $\le d$ (so essentially fixed degree) with coefficients of degree $\le n$, which is a variation of the large box model and can be handled similarly to the proof of \Cref{thm: large box}. We now fill in the details.

Given an integer $n\geq 1$ we pick elements $a_0(x), a_1(x),\ldots,a_{n-1}(x)\in\mathbb{F}_q[x]_{\le d}$ uniformly at random. For such a vector $a=(a_0(x), a_1(x),\ldots,a_{n-1}(x))\in\mathbb{F}_q[x]_{\le d}^n$ denote by
$$f_a(x,y)=y^n+a_{n-1}(x)y^{n-1}+\ldots+a_1(x)y+a_0(x)\in\mathbb{F}_q[x,y]$$ the corresponding random polynomial and consider the event:
$$A_n=\left\{a\in\mathbb{F}_q[x]_{\le d}^n: f_a(x,y)\in\F_q[x,y] \ \text{is irreducible}\right\}.$$
We are interested in $\lim\limits_{n\to\infty}\mathbb{P}(A_n)$. As indicated in the sketch above, we now switch between the variables $x$ and $y$. Write $a_i(x)=\sum\limits_{j=0}^d a_{ij}x^j$ and rewrite $f_a$ as
$$f_a= y^n+\sum\limits_{i=0}^{n-1}\left(\sum\limits_{j=0}^d a_{ij}x^j\right)y^i=\sum\limits_{j=1}^d\left(\sum\limits_{i=0}^{n-1}a_{ij}y^i\right)x^j+(y^n+a_{n-1, 0}y^{n-1}+\ldots+a_{10}y+a_{00}).$$

First, we deal with the case where $f_a$ is a constant polynomial in the variable $x$. This happens exactly when $a_0, a_1,\ldots,a_{n-1}$ are all constant polynomials. The probability of this happening is $\frac{q^n}{q^{(d+1)n}}$, which tends to $0$ as $n\to\infty$. Hence we can ignore this occurence in what follows.
Given that $f_a$ is not constant in the variable $x$, Gauss's Lemma tells us that it is irreducible in $\mathbb{F}_q[x][y]$ if and only if $\con_x(f)=1$ and $f_a$ is irreducible in $\mathbb{F}_q(y)[x]$. Hence it remains to compute
$$\lim_{n\to\infty}\P(\con_x(f)=1,\,f_a\in\F_q(y)[x]\mbox{ irreducible}).$$

First we compute the limit probability of the condition $\con_xf=1$. Denote
$$b_0=y^n+\sum_{i=0}^{n-1}a_{i0},\quad b_j=\sum_{i=0}^{n-1}a_{ij}y^i\,\,(1\le j\le d),$$ so that $f=\sum_{j=0}^db_jx^j,\,b_j\in\F_q[y]$. Note that
$b_0$ is a random element of $\mathbb{F}_q[y]^\mon_n$, and $b_j$ is a random element of $\F_q[y]_{\le n-1}$ for $1\le j\le d$ (with $b_0\ldots,b_d$ uniform and independent).
Now we can apply Proposition \ref{2.1} to the tuple $(b_0,\ldots,b_d)$ to conclude that 
\begin{equation}\label{eq: prob primitive}\lim_{n\to\infty}\P(\con_x(f)=1)=\lim_{n\to\infty}\P(\gcd(b_0,\ldots,b_d)=1)=1-\frac 1{q^d}.\end{equation}

Next we apply \Cref{thm: HIT} with the variables $x$ and $y$ switched to the (irreducible) polynomial $$F(y,T_0,T_1,\ldots,T_d,x)=(T_0+y^n)+T_1x+T_2x^2+\ldots+T_dx^d.$$
We obtain
\begin{multline}\label{eq: prof f irred hilbert}\lim_{n\to\infty}\P(f\mbox{ irreducible over }\F_q(y))=
\lim_{n\to\infty}\P(F(y,b_0-y^n,b_1,\ldots,b_d,x)\mbox{ irreducible over }\F_q(y))\\=
\lim_{n\to\infty}\frac{H_F(n-1)}{q^{dn}}=1.\end{multline}
Combining (\ref{eq: prob primitive}) and (\ref{eq: prof f irred hilbert}) we obtain
$$\lim_{n\to\infty}\P(\con_x(f)=1,\,f_a\in\F_q(y)[x]\mbox{ irreducible})=1-\frac 1{q^d}.$$
Combined with the observations above this proves the assertion of \Cref{thm: irreducibility}.

\begin{rem}\label{rem: content} It is seen from the proof above that the main source of reducible $f$ is the possibility of $\con_x(f)\neq 1$, which occurs with limit probability $\frac 1{q^d}$. Over $\Z$ if we consider a small box model $f=x^n+a_1x^{n-1}+\ldots+a_0,\,a_i\in S$ where $S$ is a fixed finite set, it appears (and can be proved conditional on the Extended Riemann Hypothesis or unconditionally for certain $S$ using the methods of \cite{BrVa19} and \cite{BKK23} respectively) that the only source of irreducible polynomials with positive limit probability is when $0\in S$ and $a_0=0$ (in which case $f$ has 0 as a root). While in our function field model we also allow $a_0=0$, this only occurs with limit probability $\frac{1}{q^{d+1}}$ and doesn't account for the bulk of reducible $f$. This is purely a function field phenomenon, related to the non-archimedean nature of function fields. A similar phenomenon was observed in the additive model studied in \cite{BEM24}.\end{rem}

\section{Small box model - the Galois group}
\label{sec: small box galois}

In the present section we continue our investigation of the small box model, this time focusing on the Galois group of $f$ and for now restricting our attention to the case of irreducible $f$. To this end we will employ the method of \cite{BaKo20}. We retain the setup and notation of \Cref{sec: small box irred}. In particular we fix a prime power $q$ and an integer $d\ge 1$, let $n$ vary and consider a random polynomial $f$ defined by (\ref{eq: rand pol 1}) with $a_i\in\F_q[x]_{\le d}$ drawn uniformly and independently. First we deal with the question of separability (as a polynomial in $y$).

\begin{lem}\label{5.1}\label{lem: irred and sep} Let $f$ be a random polynomial as above. Then
$$\lim\limits_{n\to\infty}\mathbb{P}(\text{$f$ separable over }\F_q(x) \ | \ \text{$f$ irreducible})=1.$$
Equivalently (by \Cref{thm: irreducibility}) \begin{equation}\label{eq: prob irred sep}\lim\limits_{n\to\infty}\mathbb{P}(\text{$f$ irreducible and separable over $\mathbb{F}_q(x)$})=1-\frac{1}{q^d}.\end{equation}
\end{lem}

\begin{proof} If $f$ is irreducible and inseparable then $f\in\F_q[x][y^p]$ and therefore $a_i=0$ for all $0\le i\le n-1$ not divisible by $p$. Clearly this happens with limit probability 0 as $n\to\infty$, which implies (\ref{eq: prob irred sep}) and thus the assertion.
\end{proof}




Next, we need to recall some facts about permutations. Here we will follow \cite{BaKo20}. As it turns out, the cycle structure of a random permutation in $S_n$ is similar to the decomposition of a random monic polynomial of degree $n$ in $\mathbb{F}_q[x]$ when $n$ is large. We wish to state this precisely.

\begin{notation}\label{not: X and Y} Let $\mathfrak{F}$ be the set of all infinite sequences $(m_1, m_2,\ldots)$ of non negative integers which are eventually zero, i.e $m_i=0$ for all sufficiently large $i$. For each $n\in\mathbb{N}$ and each such sequence let $X_n(m_1, m_2,\ldots)$ be the probability of a random monic polynomial of degree $n$ in $\mathbb{F}_q[x]$ having exactly $m_i$ prime factors of degree $i$ in its prime factorization over $\mathbb{F}_q$, for all $i\geq 1$. Similarly, let $Y_n(m_1, m_2,\ldots)$ be the probability of a random permutation in $S_n$ having exactly $m_i$ cycles of length $i$ in its cycle decomposition, for all $i\geq 1$. The functions $X_n$ and $Y_n$ induce probability measures on $\mathfrak F$ which by abuse of notation we also denote $X_n,Y_n$.
\end{notation}

Clearly, if $\mathbb{P}_{X_n}(m_1, m_2,\ldots)>0$ or $\mathbb{P}_{Y_n}(m_1, m_2,\ldots)>0$ then we have $m_i=0$ for all $i>n$. 

\begin{notation} Let $r\geq 1$, and let 
$\mathfrak{F}_r$ be the set of all infinite sequences 
$(m_r, m_{r+1},\ldots)$ of non negative integers which 
are eventually zero. In order to avoid confusion with 
the indices, we formally regard such a sequence as a 
function $\{r, r+1,\ldots\}\to\mathbb{Z}_{\geq 0}$ with 
finite support. Now we can push forward $X_n$ (resp. $Y_n$) to 
$\mathfrak{F}_r$ through restriction of sequences from $\mathbb N$ to $\{r,r+1,\ldots\}$: we define $X_{n,r}(m_r,
m_{r+1},\ldots)$ (resp. $Y_{n,r}(m_r,
m_{r+1},\ldots)$) to be the probability of a random monic 
polynomial of degree $n$ in $\mathbb{F}_q[x]$ (resp. a random permutation in $S_n$) having 
exactly $m_i$ factors of degree $i$ in its prime factorization over $\mathbb{F}_q$ (resp. $m_i$ cycles of length $i$), for all $i\geq r$. 
These functions 
induce probability measures on the set 
$\mathfrak{F}_r$, we denote them also by 
$X_{n,r}$ and $Y_{n,r}$ 
respectively (they are the push-forward of $X_n,Y_n$ respectively).
\end{notation}

We recall the following result from \cite{BaKo20}, which shows that the two probability measures $X_{n,r},Y_{n,r}$ are close to each other when $n$ and $r$ are large:
\begin{lem}[{\cite[Lemma 6]{BaKo20}}]\label{5.3} There exists a constant $C_q>0$ (depending only on $q$) such that for all $1\leq r\leq n$ and for all events $A\subset\mathfrak{F}_r$ the following inequality holds:
$$|\mathbb{P}_{X_{n,r}}(A)-\mathbb{P}_{Y_{n,r}}(A)|\leq\frac{C_q}{r}.$$
\end{lem}

\begin{notation} Let $\sigma\in S_n$ be a permutation, $1\leq k\leq n$ an integer. We denote by $\Psi(\sigma, k)$ is the set of all permutations in $S_n$ which can be obtained by changing the value of $\sigma$ in elements which belong to cycles of length at most $k$ in the cycle decomposition of $\sigma$. 
\end{notation}

\begin{lem}[{\cite[Lemma 9]{BaKo20}}]\label{5.4} \label{lem: change small cycles} There exists an absolute constant $\alpha>0$ such that the following 
holds: if $\sigma\in S_n$ is a random permutation then the probability that 
there exists a transitive subgroup $G\leqslant S_n$ which does not contain $A_n$
and such that $G\cap\Psi(\sigma, \floor{n^{\alpha}})\ne\emptyset$ tends to 
$0$ as $n\to\infty$. 
\end{lem}

Intuitively, this lemma says that if $n$ is large then given a random 
permutation $\sigma\in S_n$, almost surely there are no transitive subgroups 
other than $A_n$ and $S_n$ which contain elements from $\Psi(\sigma, 
\floor{n^{\alpha}})$. Thus elements of transitive groups not containing $A_n$ have special cycle structures which tend to differ significantly from those of random permutations. In what follows $\alpha$ will be the constant from Lemma \ref{lem: change small cycles}.

We denote the event described in \Cref{lem: change small cycles} by
$$F_n=\{\sigma\in S_n: \exists\ \text{transitive subgroup $A_n\not\subset G\leqslant S_n$ such that $G\cap\Psi(\sigma, 
\floor{n^{\alpha}})\neq\emptyset$}\},$$
so $\sigma\in F_n$ means that the long cycles of $\sigma$ coincide with an element of a transitive permutation group not containing $A_n$.
Let $E_n\subseteq\mathfrak{F}$ be the set of all sequences $m=(m_1, m_2,\ldots)$ such that there exists a transitive subgroup $A_n\subsetneq G\leqslant S_n$ and an element $g\in G$ which has exactly $m_i$ cycles of length $i$ for all $i\geq\floor{n^{\alpha}}+1$. Since $E_n$ consists of cycle structures of elements of $F_n$ (not necessarily all of them), we have $\P_{Y_n}(E_n)\le\frac{|F_n|}{n!}$. Thus by Lemma \ref{5.4} this probability tends to $0$ as $n\to\infty$. 



Similarly, if we let $r=\floor{n^{\alpha}}+1$ then we can denote by $E_{n,r}\subseteq\mathfrak{F}_r$ the set of all sequences $m=(m_r, m_{r+1},\ldots)$ such that there exists a transitive subgroup $A_n\not\subset G\leqslant S_n$ and an element $g\in G$ which has exactly $m_i$ cycles of length $i$ for all $i\geq r=\floor{n^{\alpha}}+1$. 
The question of whether a permutation $\sigma\in S_n$ has cycle structure in $E_n$ clearly does not depend on the cycles of $\sigma$ of length at most $\floor{n^{\alpha}}$, so we have
$$\mathbb{P}_{Y_{n,r}}(E_{n,r})=\mathbb{P}_{Y_n}(E_n).$$

Everything we have done with permutations can be also done with polynomials. 
As in the case of permutations above, if we let $r=\floor{n^{\alpha}}+1$  we have
$$\mathbb{P}_{X_n}(E_{n,r})=\mathbb{P}_{X_n}(E_n).$$
Now we can apply Lemma \ref{5.3} to obtain
$$|\mathbb{P}_{Y_n}(E_n)-\mathbb{P}_{X_n}(E_n)|=|\mathbb{P}_{Y_{n,r}}(E_{n,r})-\mathbb{P}_{X_{n,r}}(E_{n,r})|\leq\frac{C_q}{r}=\frac{C_q}{\floor{n^{\alpha}}+1}.$$
The right hand side clearly tends to $0$ as $n\to\infty$. Since we have seen that $\mathbb{P}_{Y_n}(E_n)\to 0$ as $n\to\infty$, it now follows that $\mathbb{P}_{X_n}(E_n)\to 0$ as well. We obtain the following
\begin{cor}\label{5.5} Let $\alpha$ be the constant from Lemma \ref{5.4}. Let $E_n\subseteq\mathfrak{F}$ be the subset of all sequences $m=(m_1, m_2,\ldots)\in\mathfrak{F}$ for which there exists a transitive subgroup $A_n\not\subset G\leqslant S_n$ and an element $g\in G$ which has exactly $m_i$ cycles of length $i$ for all $i\geq\floor{n^{\alpha}}+1$. Then $\lim\limits_{n\to\infty}\mathbb{P}_{X_n}(E_n)=0$. 
\end{cor}

We are ready to give the proof of Theorem \ref{thm: small box main}, which will be broken up into several parts. The main idea is to specialize the polynomial $f(x,y)$ at $x=0$ and consider the factorization of $\tilde f=f(0,y)\in\F_q[y]_{n}^\mon$, the latter being a (uniformly) random monic polynomial in $\F_q[y]_{n}^\mon$, hence having a factorization type distributed according to $X_n$. The Galois group of $f$ then contains an element (the Frobenius element) with cycle structure close to that of $\tilde f$, which by Lemma \ref{lem: change small cycles} almost surely implies that the Galois group of $f$ contains $A_n$ if $f$ is irreducible (in this case the Galois group is transitive).
\begin{prop}\label{5.6}\label{prop:towards small box galois} Let $\alpha$ be the constant from Lemma \ref{5.4}. Consider a random polynomial of the form
$$f(x,y)=y^n+a_{n-1}(x)y^{n-1}+\ldots+a_1(x)y+a_0(x)\in\mathbb{F}_q[x][y],$$
where $a_0, a_1,\ldots,a_{n-1}\in\mathbb{F}_q[x]_{\le d}$ are chosen uniformly and independently. Denote $\tilde{f}=f(0,y)\in\F_q[y]$ and write $\tilde{f}=st$, where $s,t\in \F_q[y]$ are coprime monic polynomials such that $s$ is squarefree and every irreducible factor of $t$ has multiplicity at least $2$. Now consider the following three events:
\begin{itemize}
\item[$B_n$] - The degree of $t$ is at most $\floor{n^{\alpha}}$.

\item[$C_n$] - There is no transitive subgroup $A_n\not\subset G\leqslant S_n$ which contains an element $g\in G$ such that for all $i\geq\floor{n^{\alpha}}+1$ the number of cycles of $g$ of length $i$ is equal to the number of irreducible factors of $\tilde{f}$ of degree $i$.

\item[$D_n$] - The polynomial $f$ is separable (in $y$) and irreducible. 
\end{itemize}
Then $\lim\limits_{n\to\infty}\mathbb{P}(B_n\cap C_n\cap D_n)= 1-\frac{1}{q^d}$.
\end{prop}
\begin{proof} By Lemma \ref{5.1} we know that $\lim\limits_{n\to\infty}\mathbb{P}(D_n)=1-\frac{1}{q^d}$. So it is clearly sufficient to show that $\mathbb{P}(B_n^c) ,\mathbb{P}(C_n^c)\to 0$ as $n\to\infty$. 

First of all, observe that 
$\tilde{f}$ is a uniformly random 
polynomial in $\F_q[y]_n^\mon$. Let $m$ be a natural number. The probability that $\tilde f$ is divisible by $P^2$ for any prime $P\in\F_q[y]$ with $\deg P>m$ is bounded by
$$\sum_{P\atop{\deg P>m}} \frac 1{q^{2\deg P}}\le\sum_{l=m+1}^\infty\frac{q^l}{q^{2l}}\ll \frac 1{q^m},$$ which goes to 0 as $m\to\infty$. On the other hand if $\tilde f$ is not divisible by $P^2$ for a prime with $\deg P>m$ then $\deg t$ is bounded in terms of $m$. Taking $m=m(n)$ to infinity sufficiently slowly we obtain $\mathbb P(B_n^c)\to 0$ as $n\to\infty$.


It is left to show that $\mathbb{P}(C_n^c)\to 0$ as well. The event $C_n^c$ means that there is some transitive subgroup $A_n\not\subset G\leqslant S_n$ and an element $g\in G$ such that for all $i\geq\floor{n^{\alpha}}+1$ the number of cycles of $g$ of length $i$ is equal to the number of irreducible factors of $\tilde{f}$ of degree $i$. Since $\tilde{f}$ is a random monic polynomial of degree $n$ over $\mathbb{F}_q$, the probability $\mathbb{P}(C_n^c)$ is exactly the probability $\mathbb{P}_{X_n}(E_n)$ from Corollary \ref{5.5}. By Corollary \ref{5.5} this probability indeed tends to $0$.
\end{proof}

\begin{prop}\label{5.7}\label{prop:lifting} Let $\alpha$ be the constant from Lemma \ref{5.4}. Consider the random polynomial
$$f(x,y)=y^n+a_{n-1}(x)y^{n-1}+\ldots+a_1(x)y+a_0(x)\in\mathbb{F}_q[x][y],$$
where $a_0, a_1,\ldots,a_{n-1}\in\mathbb{F}_q[x]_{\le d}$ are uniformly random. As in Proposition \ref{prop:towards small box galois} we write $\tilde{f}=f(0,y)=st$ where $s,t\in \F_q[y]$ are coprime monic polynomials such that $s$ is squarefree and every irreducible factor of $t$ has multiplicity at least $2$. Let $B_n ,C_n, D_n$ be as in Proposition \ref{prop:towards small box galois}. Finally let
$$J_n=\{\text{$f$ separable, the Galois group of $f$ over $\mathbb{F}_q(x)$ is either $A_n$ or $S_n$}\}.$$
Then $B_n\cap C_n\cap D_n\subset  J_n$. 
\end{prop}

\begin{proof} Suppose that $f\in B_n\cap C_n\cap D_n$. In particular, $f$ is separable and irreducible. By a standard lifting of Galois elements argument (e.g. by \cite[Proposition 2.8(iii)]{BEM24}) applied to the $q$-Frobenius, there exists $g\in\Gal(f/\F_q(x))$ and a surjective map $$\psi:Z_f=\left\{\gamma\in\overline{\F_q(x)}:f(\gamma)=0\right\}\to Z_{\tilde f}=\left\{\gamma\in\overline\F_q:\tilde f(\gamma)=0\right\}$$ such that $\psi(g(\gamma))=(\psi(\gamma))^q$ for all $\gamma\in Z_f$.

Since $f\in B_n\cap D_n$ the roots of $\tilde f$ with multiplicity $\ge 2$ have degree $\le k=\lfloor n^\alpha\rfloor$ over $\F_q$ and therefore the cycles of length $>k$ in $g$ correspond to the degrees $>k$ in the prime factorization of $\tilde f$. Hence the assumption $f\in C_n$ combined with the irreducibility of $f$ (so $\Gal(f/\F_q(x))$ is transitive) implies that $\Gal(f/\F_q(x))\supset A_n$ and $f\in J_n$.

\end{proof}

\begin{proof}[Proof of Theorem \ref{1.3}] Let $B_n, C_n, D_n, J_n$ be the events we have defined in Propositions \ref{prop:towards small box galois} and \ref{prop:lifting}. By Lemma \ref{lem: irred and sep}, $\lim\limits_{n\to\infty}\mathbb{P}(D_n)=1-\frac{1}{q^d}$. Also, by Proposition \ref{5.6} we have
$$\lim\limits_{n\to\infty}\mathbb{P}(B_n\cap C_n\cap D_n)=1-\frac{1}{q^d}.$$
Finally, by Proposition \ref{5.7}:
$$\mathbb{P}(B_n\cap C_n\cap D_n)\leq\mathbb{P}(D_n\cap J_n)\leq\mathbb{P}(D_n).$$
It follows that $\mathbb{P}(D_n\cap J_n)\to 1-\frac{1}{q^d}$ as $n\to\infty$, which proves Theorem \ref{thm: small box main} 
\end{proof}

\section{Small box model - distinguishing between the symmetric and alternating groups}
\label{sec: chowla}

As we noted in the introduction, Theorem \ref{1.3} is the best we can prove at the moment unconditionally. In the present section we prove Theorem \ref{1.5} which says that assuming Conjecture \ref{1.4} we can distinguish between the groups $A_n$ and $S_n$ and show that $A_n$ occurs with zero limit probability. Throughout this section assume that $q$ is odd. 

Recall that over a field of characteristic $\neq 2$ the Galois group of a degree $n$ separable polynomial is contained in $A_n$ if and only if its discriminant is a square in the field. Given our assumption that $q$ is odd, we need to show that as $n\to\infty$, the discriminant of our random polynomial $f$ in the small box model is not a square in $\mathbb{F}_q(x)$ (equivalently in $\F_q[x]$) with probability tending to $1$. Assuming Conjecture \ref{1.4}, we will be able to do this. 

In what follows we fix a prime power $q$ and $d\ge 1$ and consider a random polynomial of the form
$$f(x,y)=y^n+a_{n-1}(x)y^{n-1}+\ldots+a_1(x)y+a_0(x)\in\mathbb{F}_q[x][y],$$
where $a_0, a_1,\ldots,a_{n-1}\in\F_q[x]_{\le d}$ are drawn uniformly and independently. As in the last section, let $D_n$ be the event that $f$ is irreducible and separable in $y$. We proved that $\lim\limits_{n\to\infty}\mathbb{P}(D_n)=1-\frac{1}{q^d}$. Moreover, by Theorem \ref{1.3} we know that conditional on $D_n$, the Galois group of $f$ over $\mathbb{F}_q(x)$ is almost surely equal to $A_n$ or $S_n$ when $n$ is large. So the only thing which can prevent the Galois group from being equal to $S_n$ is that $\text{disc}_y(f)$ might be a square in $\mathbb{F}_q[x]$.  Recall that if $\alpha_1,\ldots,\alpha_n$ are the roots of $f$ in the algebraic closure of $\mathbb{F}_q(x)$ then
\begin{equation}\label{eq: disc def}\text{disc}_y(f)=\prod\limits_{i<j}(\alpha_i-\alpha_j)^2\in\mathbb{F}_q[x].\end{equation}
(we are using the fact that $f$ is monic).

If our polynomial $f(x,y)\in\mathbb{F}_q[x]$ is such that $\text{disc}_y(f)$ is a square in $\mathbb{F}_q[x]$, then for each $k\geq 1$ and $\tau\in\mathbb{F}_{q^k}$, the discriminant of $f(\tau, y)\in\mathbb{F}_{q^k}[y]$ is a square in $\mathbb{F}_{q^k}$, since $\disc_y\left(f(\tau,y)\right)=\disc_y(f(x,y))(\tau)$. So if $\disc_yf(x,y)$ is a square in $\mathbb{F}_q[x]$ then we obtain many polynomials in $\mathbb{F}_{q^k}[y]$ with a discriminant which is a square in $\F_{q^k}$. 
Our main strategy will be to relate the question of whether $\disc_y(f(\tau,y))(\tau)$ is a square to the value of the Liouville function $\lambda(f(\tau,y))$. If $f(\tau,y)$ is squarefree this can be done via Pellet's formula, see Lemma \ref{lem: pellet} below. Then we will apply Conjecture \ref{conj: chowla} to conclude that $\disc_y(f(x,y))$ cannot be a square too often. We will however need to make sure that $\disc_y(f(\tau,y))\neq 0$ for the values of $\tau$ we pick, which will require some technical work and a result of Carmon \cite{Car21}.

In what follows we fix $d\ge 1$ and consider a random polynomial of the form
\begin{equation}\label{eq: rand pol chowla} f=y^n+a_{n-1}(x)y^{n-1}+\ldots+a_1(x)y+a_0(x),\quad a_i\in\F_q[x]_{\le d}\end{equation}
(with $a_i$ drawn uniformly and independently).

\begin{lem}\label{lem: carmon cond} Let $m,k$ be fixed and $\tau_1,\ldots,\tau_m\in\F_{q^k}\setminus\F_q$ non-conjugate elements. Consider a random polynomial of the form
\begin{equation}\label{eq: f0 form}f_0=y^n+a_{n-1}(x)y^{n-1}+\ldots+a_1(x)y,\quad a_i\in\F_q[x]_{\le d}\end{equation}(with $a_i$ drawn uniformly and independently)
and define
\begin{equation}\label{eq: Nf0 def}N_{f_0}=\prod_{i=1}^{m}\prod_{j=1}^k\left(T+f_0\left(\tau_i^{q^j},y\right)\right)\in \F_q[y][T].\end{equation}
Then
$$\lim_{n\to\infty}\mathbb P(N_{f_0}\mbox{ is separable in the variable }T)=1.$$
\end{lem}

\begin{proof} By assumption the elements $\tau_i^{q^j}$ are all distinct. The polynomial $N_{f_0}(T)$ is inseparable iff the values $f_0\left(\tau_i^{q^j},y\right),\,1\le i\le m, 1\le j\le k$ are not all distinct. If $f_0,(i,j)\neq (i',j')$ are such that $f_0\left(\tau_i^{q^j},y\right)=f_0\left(\tau_{i'}^{q^{j'}},y\right)$, then $f_1=f_0+x\sum_{l=1}^{n-1}b_ly^l$ satisfies $f_0\left(\tau_i^{q^j},y\right)\neq f_0\left(\tau_{i'}^{q^{j'}},y\right)$ whenever not all $b_l$ are zero.
Since the set of polynomials $f_0$ of the form (\ref{eq: f0 form}) can be subdivided into such families (with $b_l$ ranging over $\F_q$ for $1\le l\le n-1$), it follows that $\mathbb P\left(f_0\left(\tau_i^{q^j},y\right)=f_0\left(\tau_{i'}^{q^{j'}},y\right)\right)\le q^{1-n}\to 0$ as $n\to\infty$ for any specific pair $(i,j)\neq(i',j')$. Summing over all such pairs gives the assertion.
\end{proof}

We will need the following lemma, which follows from the work of Carmon \cite{Car21} on squarefree values of polynomials over $\F_q[x]$.

\begin{lem}\label{lem: carmon} Let $k,m$ be fixed natural numbers, $n$ a variable natural number, \begin{equation}\label{eq: N degrees}F(y,T)\in\F_q[y][T],\quad\deg_TF\le km,\deg_yF\le kmn\end{equation}
separable in the variable $T$ and $h\in\F_q[y]_{\le n-1}$ uniformly random. Then we have
$$\lim_{n\to\infty}\sup_{F}\P\left(\exists P\in\F_q[y]\mbox{ prime},\deg P\ge n/2 \mbox{ such that }P^2|F(y,h(y))\right)=0,$$
where the supremum is over all separable $F$ satisfying (\ref{eq: N degrees}).
\end{lem}

\begin{proof} The assertion follows at once by combining \cite[equation (4.3)]{Car21} and \cite[equation (4.7)]{Car21}. 
\end{proof}

\begin{prop}\label{prop: sep for chowla}Let $m,k,n$ be natural numbers with $k$ prime, $\tau_1,\ldots,\tau_m\in\F_{q^k}\setminus\F_q$ elements which are non-conjugate over $\F_q$ and such that $1,\tau_i,\tau_i^2,\ldots,\tau_i^d$ are linearly independent over $\F_q$ for each $1\le i\le m$. 
Let $$f_0=y^n+a_{n-1}(x)y^{n-1}+\ldots+a_1(x)y,\quad a_i\in\F_q[x]_{\le d}$$ be such that 
$N_{f_0}$ is separable in $T$ (using the notation (\ref{eq: Nf0 def})).
Let $h\in\F_q[x]_{\le n-1}$ be a (uniformly) random polynomial. Then
\begin{equation}\label{eq: sep for chowla}\lim_{k\to\infty}\varlimsup_{n\to\infty}\sup_{f_0,\tau_1,\ldots,\tau_m}\,\mathbb P\biggl(\exists 1\le i\le m:\,f_0(\tau_i,y)+h(y)\mbox{ is not squarefree and }\con_x(f_0(x,y)+h(y))=1\biggr)=0,\end{equation}
where the supremum is over all $f_0$ and $\tau_1,\ldots,\tau_m$ as above.
\end{prop}

\begin{proof} If $f_0(\tau_i,y)+h(y)$ is not squarefree for some $i$, then there exists $\pi\in\F_{q^k}[y]$ such that $\pi^2|f_0(\tau_i,y)+h(y)$.
There are two cases to consider: if $\pi\in\F_q[y]$ then since $\pi|f_0(\tau_i,y)+h(y)$, by our assumption that $1,\tau_i,\ldots,\tau_i^d$ are linearly independent it follows that $\pi\mid\con_x(f_0(x,y)+h(y))$. To see this write $f_0(x,y)+h(y)=\sum_{l=0}^dx^lb_l(y),\,b_l\in\F_q[y]$; then $\pi|\sum_{l=0}^db_l(y)\tau_i^l$ and by the linear independence of $\tau_i^l$ (in this case $\F_{q^k}[y]$ is a free $\F_q[y]$-module with basis $1,\tau_i,\ldots,\tau_d,\rho_{1},\ldots,\rho_{k-d-1}$ for some $\rho_1,\ldots,\rho_{k-d-1}\in\F_{q^k}$) we must have $\pi|b_l$ for each $l$, so $\pi|\gcd(b_0,\ldots,b_l)=\con_x(f)$. So in this case the event we are interested in does not occur.

On the other hand if $\pi\in\F_{q^k}[y]\setminus\F_q[y]$ then since $k$ is prime we have that $P=N_{\F_{q^k}/\F_q}(\pi)=\prod_{j=1}^k\pi^{\sigma^j}\in\F_q[y]$ (here $\sigma$ is the $q$-Frobenius automorphism of $\F_{q^k}$ extended to $\F_{q^k}[y]$ by coefficient-wise action) is prime of degree $k\deg\pi$ and we have \begin{equation}\label{eq: P2 divides Nf0}P^2\,\mid\,N_{\F_{q^k}/\F_q}(f_0(\tau_i,y)+h(y))\,\mid\,
\prod_{l=1}^mN_{\F_{q^k}/\F_q}(f_0(\tau_l,y)+h(y)))=
N_{f_0}(y,h(y)).\end{equation} We will treat the cases $\deg\pi\le n/2k$ and $\deg\pi>n/2k$ separately.

\case{Small primes: $\deg\pi\le n/2k$} If $\pi^2|f_0(\tau_i,y)+h(y)$ then for $h_1\in\F_q[y]_{\le n-1}$ we have $\pi^2|f_0(\tau_i,y)+h_1(y)$ iff $\pi^2|h-h_1$ iff $P^2|h-h_1$ (we are using the fact that $k$ is prime and $\pi\in\F_{q^k}[y]\setminus\F_q[y]$). Since $\deg P^2=2k\deg\pi\le n$ it follows that $$\bigl|\{h\in\F_q[y]_{\le n-1}:\,\pi^2|f_0(\tau_i,y)+h(y)\}\bigr|=q^{n-2\deg P}$$ and hence $\P(\pi^2|f_0(\tau_i,y)+h(y))= q^{-2\deg P}$. Noting that $\deg P=k\deg\pi\ge k$ we obtain
\begin{equation}\label{eq: small primes}\P\left(\pi^2|f_0(\tau_i,y)+h(y)\mbox{ for some prime }\pi\in\F_{q^k}\setminus\F_q,\,\deg\pi\le n/2k\right)\le\sum_{P\, \mathrm{prime}\atop{k\le\deg P\le n}}q^{-2\deg P}=O(q^{-k}).\end{equation}
\case{Large primes: $\deg\pi>n/2k$} In this case $\deg P>n/2$ and by (\ref{eq: P2 divides Nf0}) and Lemma \ref{lem: carmon} applied to $N_{f_0}$ we obtain
\begin{multline}\label{eq: large primes}\P\left(\pi^2|f_0(\tau_i,y)+h(y)\mbox{ for some prime }\pi\in\F_{q^k}\setminus\F_q,\,\deg\pi>n/2k\right)\\
\le \P\left(P^2|N_{f_0}(y,h(y))\mbox{ for some prime }P\in\F_q,\,\deg\pi>n/2\right)\to 0\end{multline}
as $n\to\infty$ uniformly in $f_0$ as long as $m,k$ are fixed.

As we observed above, the event in (\ref{eq: sep for chowla}) is contained in the union of the events in (\ref{eq: small primes}) and (\ref{eq: large primes}) over all $1\le i\le m$ and the proposition follows.

\end{proof}


\begin{lem}\label{6.5}\label{lem: pellet} 
(Pellet's formula) Assume $q$ is odd. 
Let $f\in\mathbb{F}_q[y]$ be a non constant polynomial. 
Then 
$\mu(f)=(-1)^{\deg f}\chi(\disc(f))$, 
where $\chi$ is the unique quadratic character of $\F_q$.
\end{lem}
\begin{proof} See \cite[Lemma 4.1]{Con05}.
\end{proof}

We now come to the main result of this section, from which Theorem \ref{thm: chowla} will follow at once, and it will also be instrumental for the proof of Theorem \ref{thm: reducible chowla}.

\begin{thm} Assume Conjecture \ref{conj: chowla}. Let $q,d$ be fixed, with $q$ odd. \label{thm: disc} Let $f$ be a random polynomial as in (\ref{eq: rand pol chowla}). Then \begin{equation}\label{eq: need to prove 2}\lim_{n\to\infty}\P\left(\con_x(f)=1\mbox{ and }\disc_y(f)=at^2\mbox{ for some }t\in\F_q[x],\,a\in\F_q\right)=0\end{equation}
(the condition on $\disc_y(f)$ is equivalent to it being a square in $\overline \F_q[x]$).
\end{thm}

\begin{proof} 
For brevity we denote squares in $\F_q[x]$ by $\square$. Clearly it is enough to prove
$$\lim_{n\to\infty}\P(\disc_y(f)=a\cdot\square\mbox{ and }\con_x(f)=1)=0$$ for any fixed $a\in\F_q$. First assume $a=0$. It follows from (\ref{eq: prob irred sep}) combined with Proposition \ref{prop: content} that $\lim_{n\to\infty}\P(f\mbox{ irreducibe and separable})=\lim_{n\to\infty}\P(\con_x(f)=1)=1-q^{-d}$. Noting that for irreducible $f$ we have $\disc_y(f)=0$ iff $f$ is inseparable, we obtain (\ref{eq: need to prove 2}). So for the rest of the proof we fix $a\in\F_q^\times$.

For now we fix a natural number $m$ and a prime $k>2$. Write $f=f_0+h$, where $f_0$ is a random polynomial as in Lemma \ref{lem: carmon cond} and $h\in\F_q[y]_{\le n-1}$ uniformly. First we note that by Lemma \ref{lem: carmon cond}, asymptotically almost surely $N_{f_0}$ is separable. So it is enough to show
$$\lim_{n\to\infty}\P(\disc_y(f)=a\cdot\square,\,N_{f_0}\mbox{ is separable},\,\con_x(f)=1)=0$$ and for this it is enough to show that for any given $f_0$ with $N_{f_0}$ separable, we have \begin{equation}\label{eq: need to prove}\lim_{n\to\infty}\P\left(\disc_y(f_0+h)=a\cdot\square,\,\con_x(f_0+h)=1\right)=0\end{equation}uniformly in $f_0$. From now on we assume that $f_0$ is given and $h\in\F_q[x]_{\le n-1}$ is uniformly random.

Now we make the assumption that $k\ge k_0(q,d,m)$ is large enough so that there exist non-conjugate $\tau_1,\ldots,\tau_m\in\F_{q^k}\setminus\F_q$ such that $1,\tau_i,\tau_i^2,\ldots,\tau_i^d$ are linearly independent over $\F_q$ for each $1\le i\le m$ (so they satisfy the assumptions of Proposition \ref{prop: sep for chowla}). 

Let $h$ be such that 
\begin{enumerate}\item[(i)]$\disc_y(f_0+h)=a\cdot\square$ and $\con(f_0+h)=1,$
\item[(ii)]$f_0(\tau_i,y)+h(y)$ is squarefree for each $1\le i\le m$.
\end{enumerate} (by Proposition \ref{prop: sep for chowla}, (i) implies (ii) with limit probability close to 1 if $k$ is large). By Pellet's formula (Lemma \ref{lem: pellet}) and condition (ii) we have (noting that the Liouville function $\lambda$ and the M\"obius function $\mu$ coincide on squarefree arguments) \begin{equation}\label{eq: lambda mu}\lambda(f_0(\tau_i,y)+h(y))=\mu(f_0(\tau_i,y)+h(y))=(-1)^n\chi_k(\disc_y(f_0(\tau_i,y)+h(y)))\quad(1\le i\le m),\end{equation} where $\chi_k$ is the unique quadratic character of $\F_{q^k}$. Then by condition (i) for each $1\le i\le m$ we have that $\disc_y(f_0(\tau_i,y)+h(y))=as_i^2$ for some $s_i\in\F_{q^k}$ and we get (using (\ref{eq: lambda mu}) and the fact that the $q$-Frobenius is an automorphism of $\F_{q^k}$) \begin{equation}\label{eq: lambda eq1} \lambda(f_0(\tau_i^{q^j},y)+h(y))=(-1)^n\chi_k(a)=(-1)^n\chi(a),\quad 1\le i\le m,\,1\le j\le k,\end{equation} 
where $\chi$ is the unique quadratic character of $\F_q$ (we are using the fact that $k$ is an odd prime). Denote $$N_i=\prod_{j=1}^k(T+f_0(\tau_i^{q^j}))\in\F_q[y][T],\quad 1\le i\le m.$$ Note that $N_{f_0}=N_1\cdots N_m$, hence the $N_i$ are separable and pairwise coprime (we are assuming that $N_{f_0}$ is separable). Taking the product of (\ref{eq: lambda eq1}) over all $1\le j\le k$ and using the multiplicativity of $\lambda$ and that $k$ is odd we obtain $\lambda(N_i(y,h(y)))=(-1)^n\chi(a)$ for all $1\le i\le m$ and hence
\begin{equation}\label{eq: key product}\frac 1{2^m}\prod_{i=1}^m \bigl(1+(-1)^n\chi(a)\lambda(N_i(y,h(y)))\bigr)=1.\end{equation}
(note that the value $\lambda(g)$ for $g\in\F_q[y]$ is the same whether we view $g$ as a polynomial in $\F_q[y]$ or $\F_{q^k}[y]$, because $k$ is an odd prime and so any prime of $\F_q[y]$ splits into an odd number of primes in $\F_{q^k}[y]$).
Note that the LHS of (\ref{eq: key product}) is nonnegative even if $h$ does not satisfy the conditions (i),(ii) above. Hence if we sum (\ref{eq: key product}) over all $h\in\F_q[y]_{\le n-1}$ and divide by $q^n$ we obtain (using the multiplicativity of $\lambda,\chi$ for the last equality)
\begin{multline*}\P(\mbox{(i),(ii) hold})\le \frac 1{2^mq^n}\sum_{h\in\F_q[y]_{\le n-1}}\prod_{i=1}^m\bigl(1+(-1)^n\chi(a)\lambda(N_i(y,h(y)))\bigr)\\
=\frac 1{2^m}+\frac 1{2^m}\sum_{\emptyset\neq S\subset\{1,\ldots,m\}}(-1)^{n+|S|}\chi(a)^{|S|}\frac 1{q^n}\sum_{h\in\F_q[y]_{\le n-1}}\lambda\left(\prod_{i\in S}N_i(y,h(y))\right).\end{multline*}

Fix $\epsilon>0$. Since $N_i(y,T)$ are separable and pairwise coprime, each of the products $ F_S=\prod_{i\in S}N_i(y,T)$ is separable. Moreover $\deg_T F_S\le km,\,\deg_y F_S\le kmn$. Hence we may apply Conjecture \ref{conj: chowla} (which we are assuming for this proof) to $F_S$ to obtain that for $n\ge n_0(q,d,m,k)$ we have $$\frac 1{q^n}\sum_{h\in\F_q[y]_{\le n-1}}\lambda\left(\prod_{i\in S}N_i(y,h(y))\right)<\epsilon$$ and so
$\P(\mbox{(i),(ii) hold})\le \frac 1{2^m}+\epsilon$. By Proposition \ref{prop: sep for chowla}, for large enough $k\ge k_1(q,d,m)$ (and we may take $k_1(q,d,m)\ge k_0(q,d,m)$) and $n\ge n_1(q,d,m,k)$ (we may take $n_1(q,d,m,k)\ge n_0(q,d,m,k)$) we have $\P(\mbox{(i) holds})\le\P(\mbox{(i),(ii) hold})+\epsilon\le\frac 1{2^m}+2\epsilon$.

Since we can take $m$ as large as we wish and $\epsilon$ as small as we wish it follows that $\lim_{n\to\infty}\P(\mbox{(i) holds})=0$ uniformly in $f_0$, i.e. we have (\ref{eq: need to prove}), which as we observed implies the assertion of the theorem.
\end{proof}

\begin{proof}[Proof of Theorem \ref{thm: chowla}]
Let $f$ be a random polynomial as in Theorem \ref{thm: chowla} and denote by $G_f$ its Galois group. By Theorem \ref{thm: small box main} we have $\lim_{n\to\infty}\P(G_f=S_n\mbox{ or }A_n\,|\,f\mbox{ irreducible})= 1$ as $n\to\infty$. However by Theorem \ref{thm: disc} we have (assuming Conjecture \ref{conj: chowla}) $\P(G_f=A_n)\le\P(\disc_y(f)=\square,\,\con_x(f)=1)\to 0$ as $n\to\infty$, since $G_f=A_n$ implies that $\disc_y(f)=\square$ and $f$ is irreducible, so in particular $\con_x(f)=1$. Since by Theorem \ref{thm: irreducibility} we have $\lim_{n\to\infty}\P(f\mbox{ irreducible})=1-q^{-d}>0$, it follows that $$\lim_{n\to\infty}\P(G_f=S_n\,|\,f\mbox{ irreducible})=\lim_{n\to\infty}\P(G_f=S_n\mbox{ or }A_n\,|\,f\mbox{ irreducible})=1.$$
\end{proof}

\section{The reducible case}
\label{sec: reducible}
So far we have studied the probability of polynomials being irreducible, and described their typical Galois groups in that case. Now we consider the reducible case. In the present section we work with the small box model, so we fix a prime power $q$ and $d\ge 1$ throughout this section and consider the random polynomial
$$f(x,y)=y^n+a_{n-1}(x)y^{n-1}+\ldots+a_1(x)y+a_0(x),\quad a_i\in\F_q[x]_{\le d},$$
with $a_i$ drawn uniformly and independently. 
In the small box model (i.e when $d$ is fixed and we take $n\to\infty$) the polynomial is reducible with positive probability, so now we wish to study its factors.

\subsection{Content and primitive part}
 
What we have seen in the course proving Theorem \ref{1.2} is that if $\con_x(f)=1$ then $f$ is asymptotically almost surely irreducible as $n\to\infty$. Now we have to deal with the case $\con_x(f)\neq 1$. We will be able to reduce this case to the primitive case $\con_x(f)=1$.
Indeed, let us write $f(x,y)=c(y)g(x,y)$ where $c(y)=\con_x(f)$ is monic of degree $k$ and $g(x,y)\in\mathbb{F}_q[x][y]$ is a monic polynomial of $y$-degree $n-k$ which is content-free in the variable $x$ (i.e. $\con_x(g)=1$). We call $g$ the \emph{primitive part} of $f$. If $n$ is large and $k$ is small (the latter usually holds because the GCD of more than one random polynomial tends to be small), our results on the content-free case will show that $g(x,y)$ is almost surely irreducible. As we will see, asymptotically almost surely the Galois group of $f$ is determined by the Galois groups of $g$ and $c$ (up to the ambiguity described in Theorem \ref{thm: reducible}). Since $g$ is content-free we already have significant information about the former, and the latter is easy to handle. In what follows we make this intuition precise.

\begin{lem}\label{7.1}\label{lem: content deg}
Assume $0\le k<n$ then $$\P(\deg_y(\con_x(f))=k)=\left(1-\frac 1{q^d}\right)\cdot\frac 1{q^{dk}}.$$
\end{lem}
\begin{proof} By Proposition \ref{prop: content} we have
$$\P(\deg(\con_x(f))=k)=\sum_{c\in\F_q[y]\atop{\mathrm{monic}\atop\deg c=k}}\left(1-\frac 1{q^d}\right)\frac 1{q^{(d+1)k}}=q^k\left(1-\frac 1{q^d}\right)\frac 1{q^{(d+1)k}}=\left(1-\frac 1{q^d}\right)\frac 1{q^{dk}}.$$
\end{proof}

\begin{prop}\label{7.2}\label{prop: prim irreducible} $\lim_{n\to\infty}\P(\mbox{the primitive part of }f\mbox{ is irreducible})=1$.
\end{prop}
\begin{proof} 
Let $m$ be a fixed natural number to be chosen later. Note that since the decomposition $f(x,y)=c(y)g(x,y)$ with $c=\con_x(f)$ and $\con_x(g)=1$ is unique, conditional on $\deg(\con_xf)=k$ ($0\le k\le n-1$) the primitive part of $f$ is uniformly distributed in the set of polynomials of the form $y^{n-k}+\sum_{i=0}^{n-k-1}b_iy^i,\,b_i\in\F_q[x]_{\le d}$ and moreover is independent of $\con_x(f)$. Hence denoting by $E_k$ the event that such a random polynomial is irreducible and by $E$ the event in the LHS of the assertion, we have (using the fact that $\P(\deg(\con_x(f))=n))=\P(f\in\F_q[y]_n^\mon)=q^{-dn}$)
\begin{multline*}\P(E)\leq\P(\deg(\con_x(f))=n)+\sum_{k=0}^{n-1}\P(\deg(\con_x(f))=k)\P(E_k)
=q^{-dn}+\sum_{k=0}^{n-1}\left(1-\frac 1{q^d}\right)\frac 1{q^{dk}}\P(E_k)\\
\le q^{-dn}+\sum_{k=0}^m\left(1-\frac 1{q^d}\right)\frac 1{q^{dk}}\P(E_k)+\sum_{k=m+1}^{n-1}\left(1-\frac 1{q^d}\right)\frac 1{q^{dk}}.
\end{multline*}
Noting that the first summand in the last expression tends to 0 as $n\to\infty$, the second summand also tends to 0 for any fixed $m$ by Theorem \ref{thm: irreducibility} and the third summand can be made arbitrarily small by choosing $m$ large enough, we obtain $\lim_{n\to\infty}\P(E)=0$.

\end{proof}
\subsection{The Galois group}
Now we would like to study the Galois group of the splitting field of our random polynomial $f$ over $\mathbb{F}_q(x)$. By Theorem \ref{1.3} we know that given that $f$ is irreducible, then with probability tending to $1$ as $n\to\infty$ it is separable and its Galois group is $A_n$ or $S_n$. At present we have no way of distinguishing between these two groups without assuming Conjecture \ref{1.4}. Similarly, in the reducible case we will obtain a few possibilities for the Galois group, however we will not be able to distinguish between some of them without assuming Conjecture \ref{1.4}. 

Write $f(x,y)=c(y)g(x,y)$, with $c=\con_x(f)$ and $g$ the primitive part of $f$. We will explore the Galois group of $f$ over $\mathbb{F}_q(x)$ through the Galois groups of the splitting fields of $c$ and $g$. First, note that it does not matter whether $c$ is separable or not, its splitting field over $\mathbb{F}_q(x)$ is a Galois extension of $\mathbb{F}_q(x)$ of the form $\mathbb{F}_{q^r}(x)$. Its Galois group is cyclic and generated by the automorphism acting as the $q$-Frobenius coefficient-wise.

It remains to understand what is the splitting field of $g(x,y)$ over $\mathbb{F}_q(x)$ and how it interacts with that of $c$. A priori we do not know that the splitting field is a Galois extension of $\F_q(x)$, unless the polynomial $g$ is separable. However, that happens almost surely, as we will now show. 
\begin{prop}\label{7.4}\label{prop: Gal g} Writing $f=cg,\,c=\con_x(f)$ we have $$\lim_{n\to\infty}\P\left(g\mbox{ is separable and }\Gal(g/\F_q(x))=S_{\deg g}\mbox{ or }A_{\deg g}\right)=1.$$ 
\end{prop}

\begin{proof} The proof is the same as the proof of Proposition \ref{prop: prim irreducible}, except one invokes Theorem \ref{thm: small box main} instead of Theorem \ref{thm: irreducibility}.
\end{proof}

We are now ready to prove the main results of this section - Theorems \ref{thm: reducible} and \ref{thm: reducible chowla}.

\begin{proof}[Proof of Theorem \ref{1.6}] 
Denote $k=\deg c$. First we note that by Proposition \ref{prop: Gal g}, with limit probability 1 the polynomial $g$ is separable with $\Gal(g)=S_{n-k}$ or $A_{n-k}$. By Lemma \ref{lem: content deg} with limit probability 1 we also heve $n-k=\deg g\ge 5$. Hence it is enough to show that if $g$ is separable, $\Gal(g)=S_{n-k}$ or $A_{n-k}$ and $n-k\ge 5$ then $\Gal(f)$ is one of the groups (a),(b),(c) described in Theorem \ref{thm: reducible}. 

Henceforth we assume $g$ is separable, $\Gal(g)=S_{n-k}$ or $A_{n-k}$ and $n-k\ge 5$. Denote by $M_1$ the splitting field of $c(y)$ over $\mathbb{F}_q(x)$, and by $M_2$ the splitting field of $g(x,y)$ over $\F_q(x)$ (we are assuming $M_1,M_2$ are contained in a common algebraic closure of $\F_q(x)$). As noted before, we have $M_1=\mathbb{F}_{q^r}(x)$ for some $r\geq 1$, and it is a Galois extension of $\mathbb{F}_q(x)$ with cyclic Galois group $C$ of order $r$. Since $g$ is separable, $M_2$ is also a Galois extension of $\mathbb{F}_q(x)$. The splitting field of $f=cg$ over $\mathbb{F}_q(x)$ is the compositum $L=M_1M_2$. This is also a Galois extension of $\mathbb{F}_q(x)$, as a compositum of Galois extensions. Thus we see that the splitting field of $f$ is a Galois extension of $\F_q(x)$ and we wish to compute its Galois group. 

First of all, there is a natural map:
$$\varphi:\text{Gal}(L/\mathbb{F}_q(x))\to\text{Gal}(M_1/\mathbb{F}_q(x))\times\text{Gal}(M_2/\mathbb{F}_q(x))$$
given by $\sigma\mapsto (\sigma|_{M_1}, \sigma|_{M_2})$. This is clearly an injective group homomorphism, and further its projection to each factor $\Gal(M_i/\F_q(x))$ is surjective. The group $C=\text{Gal}(M_1/\mathbb{F}_q(x))\cong\Z/r\Z$ is the Galois group of $c(y)$ over $\mathbb{F}_q(x)$.  Also, the group $\text{Gal}(M_2/\mathbb{F}_q(x))$ is the Galois group of $g$ over $\mathbb{F}_q(x)$, which by assumption is isomorphic to $S_{n-k}$ or $A_{n-k}$. Hence $G=\Gal(f/\F_q(x))=\text{Gal}(L/\mathbb{F}_q(x))$ is isomorphic to a subgroup $H\leqslant C\times A_{n-k}$ or $H\leqslant C\times S_{n-k}$ which projects surjectively onto each factor. It remains to show that $H$ is one of the groups (a),(b),(c) described in the statement of Theorem \ref{thm: reducible}.

First consider the case $\Gal(g)=A_{n-k}$. Recall that $n-k\ge 5$, so $A_{n-k}$ is simple. Denote by $p_1:H\to C,\,p_2:H\to A_{n-k}$ the obvious projections and $N_1=p_1(\ker(p_2)),N_2=p_2(\ker(p_1))$. By Goursat's lemma (see \cite[Theorem 2]{CW75}) the image of $H$ in $C/N_1\times A_{n-k}/N_2$ is the graph of a group isomorphism $C/N_1\xrightarrow{\sim}A_{n-k}/N_2$. Since $A_{n-k}$ is simple non-abelian, we must have $N_1=C,N_2=A_{n-k}$ and hence $H=C\times A_{n-k}$, i.e. case (b) in Theorem \ref{thm: reducible}.

Next assume $\Gal(g)=S_{n-k}$ and once again denote by $p_1:H\to C,\,p_2:H\to S_{n-k}$ the obvious projections and $N_1=p_1(\ker(p_2)),N_2=p_2(\ker(p_1))$. Once again by Goursat's lemma the image of $H$ in $C/N_1\times S_{n-k}/N_2$ is the graph of a group isomorphism $C/N_1\xrightarrow{\sim}S_{n-k}/N_2$. 

If $N_1=C,N_2=S_{n-k}$ then $H=C\times S_{n-k}$, i.e. case (a) in Theorem \ref{thm: reducible}. The only other normal subgroup of $S_{n-k}$ with abelian quotient is $N_2=A_{n-k}$ and then $r=|C|$ is even and $N_1\leqslant C$ is the unique subgroup of $C$ of order $r/2$. In this case $H$ is the group described in case $(c)$ of Theorem \ref{thm: reducible}. This completes the proof.

\end{proof}

\begin{proof}[Proof of Theorem \ref{thm: reducible chowla}]
The proof is similar to the proof of Theorem \ref{thm: reducible} above, so we adopt the setup and notation of the above proof and only explain the steps which are different.

First of all since we are assuming that $q$ is odd and Conjecture \ref{conj: chowla}, we can apply the result of Theorem \ref{thm: chowla} instead of Theorem \ref{thm: small box main} and thus assume that $\Gal(g/\F_q(x))=S_{n-k}$ (with $n-k\ge 5$) and moreover, using Theorem \ref{thm: disc}, that $\disc_y(g)$ is not a square in $\overline \F_q[x]$. Thus the Galois group $G=\Gal(f/\F_q(x))$ is isomorphic to a subgroup $H\leqslant C\times S_{n-k}$ projecting surjectively on each factor.

This leaves only two cases: $H=C\times S_{n-k}$, giving the assertion of the theorem; and the second case $N_2=A_{n-k}$, where $N_2=p_2(\ker(p_1))\cong\Gal(g/M_1)=\Gal(g/\F_{q^r}[x])$. However the latter Galois group cannot be contained in $A_{n-k}$ by the discriminant criterion, since by our assumption above $\disc_y(g)$ is not a square in $\overline \F_q(x)$. We obtain a contradiction, thus eliminating this case. We are left with the case $\Gal(f/\F_q(x))\cong C\times S_{n-k}$ as required.

\end{proof}

\bibliography{mybib}
\bibliographystyle{alpha}

\end{document}